\documentclass[11pt]{amsart}

\usepackage{amssymb}
\usepackage{graphics,color}

\newtheorem{theorem}{Theorem}[section]

\newtheorem{lemma}[theorem]{Lemma}

\newtheorem{corollary}[theorem]{Corollary}
\newtheorem{remark}[theorem]{Remark}
            
\newtheorem{example}[theorem]{Example} 

\def\diag{\mathop{\mathrm{diag}}} 

\begin{document} 


\title[Cayley transform and the Kronecker product]
      {Cayley transform and the Kronecker product of Hermitian matrices}
      
\author{Yorick Hardy}
\address{Yorick Hardy, 
Department of Mathematical Sciences, University of South Africa, 
Pretoria, South Africa}
\email{hardyy@unisa.ac.za}

\author{Ajda Fo\v sner}
\address{Ajda Fo\v sner, 
Faculty of Management, University of Primorska,
Cankarjeva 5, SI-6104 Koper, Slovenia}
\email{ajda.fosner@fm-kp.si}

\author{Willi-Hans Steeb}
\address{Willi-Hans Steeb, 
International School for Scientific Computing, University of Johannesburg, 
Auckland Park 2006, South Africa}
\email{steebwilli@gmail.com}

\date{}

\begin{abstract}
We consider the conditions under which the Cayley transform of the Kronecker product of two Hermitian matrices 
can be again presented as a Kronecker product of two matrices and, if so, if it is a product
of the Cayley transforms of the two Hermitian matrices. We also study the related question: given
two matrices, which matrix under the Cayley transform yields the Kronecker product of their Cayley transforms.
\end{abstract}

\maketitle

\noindent
{\em 2010 Math. Subj. Class.}: 15A69, 15B57.

\vskip 24pt

\noindent
{\em Key words}: Cayley transform, Hermitian matrix, Kronecker product.

\vskip 24pt

\section{Introduction}

Let $M_n$ be the algebra of all $n\times n$ matrices over the complex field and $H_n\subset M_n$ the subalgebra of Hermitian matrices.
As usual, a conjugate transpose of a complex matrix $A\in M_n$ will be denoted by $A^*$. Now, suppose that $A\in H_n$, i.e., $A^*=A$, and let $I_n$ 
be the $n \times n$ identity matrix. Then $(A+iI_n)^{-1}$ exists and $$U_A=(A-iI_n)(A+iI_n)^{-1}$$ is called a Cayley transform of $A$. It is easy
to see that $U_A$ is a unitary matrix and the inverse transform is given by $$A=i(I_n+U_A)(I_n-U_A)^{-1}.$$ Furthermore, $+1$ cannot be an eigenvalue of $U_A$.
In the following we give some basic examples.

\begin{example}\em

\-

\begin{enumerate}
	\item If $A=I_n$, then $U_A=-iI_n$.
	\item If $A$ is the $n \times n$ zero matrix, i.e., $A=0_n$, then $U_A=-I_n$.
	\item If $A$ is a diagonal matrix, then $U_A$ is also a diagonal matrix. 
	\item If $A$ has degenerate eigenvalues, then $U_A$ has degenerate eigenvalues as well. 
	\item If $A\in H_n$ is unitary, i.e., $A^2=I_n$, then $U_A=-iA$.
	      The Pauli matrices $\sigma_1$, $\sigma_2$ and $\sigma_3$ satisfy these conditions.
\end{enumerate}
\end{example}

The Cayley transform is actually a generalization of a mapping of the complex plane to itself, given by $$U(z)=\frac{z-i}{z+i},\qquad z\in \mathbb C\setminus\{-i\}.$$ 
In particular, $U$ maps the upper half plane of $\mathbb C$ conformally onto the unit disc of $\mathbb C$ and the real line $\mathbb R$ injectively into 
the unit circle. Moreover, no finite point on the real line can be mapped to $+1$ on the unit circle.

\vspace{0,2cm}

Let us continue with some useful properties of the Cayley transform.
\begin{enumerate}
	\item If $V\in M_n$ is invertible, then $U_{VAV^{-1}} = VU_AV^{-1}$ for $A\in H_n$.
	\item If $\mathbf{x}\in\mathbb{C}^n$ is an eigenvector for an eigenvalue $\lambda\in\mathbb{R}$ of a matrix $A\in H_n$, then $\mathbf{x}$
	      is an eigenvector of a Cayley transform $U_A$ and $U(\lambda)=(\lambda-i)/(\lambda+i)$ is its eigenvalue.
	\item If $A\in H_m$ and $B\in H_n$, then $U_{B\otimes A}=PU_{A\otimes B}P^t$, where $P\in M_{mn}$ is the permutation matrix satisfying $P(A\otimes B)P^t=B\otimes A$. Here, 
	      $P^t$ denotes the transpose of a matrix $P$ and $\otimes$ denotes the Kronecker product (see, for example, \cite{10, 4}).
	\item If $A,B\in H_m$ such that $[A,B]=AB-BA=0$ then $[U_A,U_B]=0$.
\end{enumerate}

\vspace{0,2cm}

The Cayley transform is named after Arthur Cayley (see \cite{1,2}). In the last few decades, a lot of results about the Cayley transform and its applications, 
mostly in mathematics and physics, have been obtained. 
For example, Calixto and Perez-Romero \cite{CPR} utilized the Cayley transform for a complex Minkowski space. 
Jadczyk \cite{J} applied the Cayley transform in the compactification of the Minkowski space.
Furthermore, Eisner and Zwart \cite{EZ} studied $C_0$-semigroups and the Cayley transform. 
Noncommutative Cayley transforms have been introduced by Popescu \cite{P} and an application of the Cayley transform for rotation of elasticity tensors has been studied by Norris \cite{N}.

\vspace{0,2cm}

In mathematical physics, the main applications of the Cayley transform is to the Hermitian matrix
$$H=\left(\begin{matrix}
a+d & b-ic\\
b+ic & a-d
\end{matrix}\right)$$
with $a,b,c,d$ real. Here a question is: {\em What is the condition on $a,b,c,d$ such that the matrix $H$ and
the Cayley transform $U_H$ coincide (perhaps up to a phase)?} Note that the eigenvalues of $H$ are
$$\lambda_{1,2} = a \pm \sqrt{b^2+c^2+d^2}$$ and the eigenvalues of $U_H$ are
$$\xi_{1,2}=\frac{\lambda_{1,2} - i}{\lambda_{1,2} + i} = \frac{\lambda_{1,2}^2 - 1}{\lambda_{1,2}^2 + 1} - 2i\frac{\lambda_{1,2}}{\lambda_{1,2}^2 + 1}~.$$
Now, it is easy to see that the eigenvalues of $H$ and the eigenvalues of the Cayley transform $U_H$ 
never coincide since $\lambda_{1,2}$ are real (if $\lambda_{1,2}=0$, then $H=0_2$ and $U_H=-I_2$). On 
the other hand, if the eigenvalues of $H$ and the eigenvalues of $U_H$ differ only by a phase, then 
$$\bigg|\frac{\lambda_{1,2} - i}{\lambda_{1,2} + i}\bigg| = |\lambda_{1,2}|.$$
This yields that $\lambda_{1,2}=\pm 1.$ In each case the phase difference is $-i$ since $U(\lambda_{1,2})=-i\lambda_{1,2}$, i.e., $U_H=-iH$. 
In particular, one of the following holds:
\begin{enumerate}
	\item $a = 1$ and $b=c=d = 0$,
	\item $a = -1$ and $b=c=d= 0$,
	\item $a = 0$ and $a^2 + b^2 + c^2 = 1$.
\end{enumerate}

\vspace{0,2cm}

In the paper, we discuss the following question. Let $A\in H_m$ and $B\in H_n$ be two  Hermitian matrices. Then the Cayley transform provides the unitary matrices 
$U_A$ and $U_B$, respectively. Now, $A \otimes B$ is again a Hermitian matrix and the Cayley transform gives us another  unitary matrix $U_{A \otimes B}$.

\begin{enumerate}
	\item {\em Does there exist a map $g: H_m\times H_n\to H_m\otimes H_n$ such that for all $A\in H_m$ and $B\in H_n$ we have $U_{g(A,B)} = U_A\otimes U_B$ 
	(where it makes sense)?}
	\item {\em What is the condition on $A$ and $B$ such that $U_{A \otimes B}$ can be again presented as a Kronecker product of two matrices?}
	\item {\em What is the condition on $A$ and $B$ such that $U_{A \otimes B}$ can be presented as a Kronecker product of the Cayley transforms of two Hermitian matrices?}
	\item {\em What is the condition on $A$ and $B$ such that $U_{A \otimes B}=U_A\otimes U_B$?}
\end{enumerate}

The above questions are part of a more general question:
For a given map $f : M_m\times M_n \to M_m\otimes M_n$,
we may ask what are the conditions on $A\in M_m$ and $B\in M_n$
such that $f(A \otimes B) = f(A) \otimes f(B)$?
For example, we know that $\exp(A \otimes B) \ne \exp(A) \otimes \exp(B)$ in general.
However, $\exp(A \otimes I_n + I_n \otimes B) \equiv \exp(A) \otimes \exp(B)$
(this relates to the first question).

\vspace{0,2cm}


To conclude our introduction, we give some examples of Hermitian matrices
for each of the last three questions above. First we consider two Hermitian matrices
$A$ and $B$ such that the Cayley transform $U_{A\otimes B}$ cannot be presented
as a Kronecker product of two complex matrices.

\begin{example}\em
\label{Ex1}
Let $ A=B=\diag (1,0)=\left(\begin{smallmatrix}
1 & 0\\
0 & 0
\end{smallmatrix}\right)$ be diagonal $2\times 2$ Hermitian matrices. Then $U_{A\otimes B}=\diag (-i,-1,-1,-1)$
is a diagonal $4\times 4$ matrix which cannot be presented as a Kronecker product of two $2\times 2$ complex matrices.
\end{example}

Now, let us write one simple example showing that $U_{A\otimes B}=U_A\otimes U_B$ 
does not hold in general. In this example the Cayley transform $U_{A\otimes B}$ can
be presented as a Kronecker product of two complex matrices.

\begin{example}\em
\label{Ex2}
Let $A=0_m$ be the $m\times m$ zero matrix and $B=0_n$ be the $n\times n$ zero matrix.
Then $U_A=-I_m$, $U_B=-I_m$, $U_{A\otimes B}=-I_{mn}=U_{I_m}\otimes U_{I_n}$, and $U_A\otimes U_B=I_{mn}$.
Obviously, $U_{A \otimes B} \ne U_A \otimes U_B$.
\end{example}

Finally, we give four simple examples of pairs of $2\times 2$ Hermitian matrices $(A,B)$ satisfying 
$U_{A\otimes B}=U_A\otimes U_B$. 

\begin{example}\em
\label{Ex3}

\-

\begin{enumerate}
  \item $A = -I_2$ and  $B = (1\pm \sqrt{2}) I_2$
	\item $A = -I_2$ and  $B = \diag(1\pm \sqrt{2}, 1\mp\sqrt{2})$
\end{enumerate}
\end{example}

\section{Cayley transform on Hermitian matrices}

First we answer the question: {\em Does there exist a map $g: H_m\times H_n\to H_m\otimes H_n$ such that for all $A\in H_m$ and $B\in H_n$ we have 
$U_{g(A,B)} = U_A\otimes U_B$?}
The answer is positive on the domain where this question makes sense.
Under exponentiation of Hermitian matrices, the Kronecker sum arises naturally
as the unique $f: H_m\times H_n\to H_m\otimes H_n$
satisfying
\begin{equation*}
 \forall A\in H_m,\,B\in H_n:\qquad e^{f(A,B)} = e^A\otimes e^B.
\end{equation*}
The Kronecker sum is given by
\begin{equation*}
 f(A,B) := A\otimes I_n + I_m\otimes B.
\end{equation*}
We seek an analogue for the Cayley transform.
The domain for this problem is
\begin{equation*}
 H'_{m,n} := \{\,(A,B)\in H_m\times H_n\,:\,\text{$U_A\otimes U_B$ does not have 1 as an eigenvalue}\,\}.
\end{equation*}

\begin{theorem}
The function $g: H'_{m,n}\to H_m\otimes H_n$
\begin{equation*}
 g(A,B) = if(U_A^*,-U_B)^{-1}f(U_A^*,U_B) \equiv if(-U_A,U_B^*)^{-1}f(U_A,U_B^*)
\end{equation*}
uniquely satisfies
\begin{equation*}
 \forall (A,B)\in H'_{m,n}:\qquad U_{g(A,B)} = U_A\otimes U_B.
\end{equation*}
\end{theorem}
\begin{proof}
Suppose that a map $g: H'_{m,n}\to H_m\otimes H_n$ satisfies
\begin{equation*}
 \forall (A,B)\in H'_{m,n}:\qquad U_{g(A,B)} = U_A\otimes U_B.
\end{equation*}
This equation provides
\begin{align*}
 (g(A,B)-iI_{mn})(g(A,B)+iI_{mn})^{-1} = (U_A\otimes I_n)(I_m\otimes U_B).
\end{align*}
Therefore,
\begin{align*}
 g(A,B) &= i(U_A^*\otimes I_n-I_m\otimes U_B)^{-1}(U_A^*\otimes I_n+I_m\otimes U_B)\\
	&= if(U_A^*,-U_B)^{-1}f(U_A^*,U_B).
\end{align*}
Obviously $g(A,B)$ is uniquely determined since the expression was derived using
only invertible algebraic operations (the existence of $f(U_A^*,-U_B)^{-1}$
is demonstrated below). Similarly,
\begin{align*}
 (g(A,B)-iI_{mn})(g(A,B)+iI_{mn})^{-1} = (I_m\otimes U_B)(U_A\otimes I_n)
\end{align*}
and, thus,
\begin{equation*}
 g(A,B) = if(-U_A,U_B^*)^{-1}f(U_A,U_B^*).
\end{equation*}
The matrices $f(U_A^*,-U_B)$ and $f(-U_A,U_B^*)$ are invertible
since 1 is not an eigenvalue of $U_A\otimes U_B$.
Namely, if $x$ is an eigenvalue of $U_A$ and $y$ is an
eigenvalue of $U_B$, then the eigenvalues of $f(U_A^*,-U_B)$ are
of the form $x^*-y$. We have
\begin{equation*}
 xy\neq 1
 \quad\Longleftrightarrow\quad
 xy \neq xx^*
 \quad\Longleftrightarrow\quad
 x(x^*-y)\neq 0
 \quad\Longleftrightarrow\quad
 (x^*-y)\neq 0
\end{equation*}
since $xx^*=|x|=1$ and $x\neq 0$. Similarly we can show that $f(-U_A,U_B^*)$
is invertible on $H'_{m,n}$.

Finally, we show that $g(A,B)\in H_m\otimes H_n$. Note that $f(U_A,U_B^*)$
and $f(U_A,-U_B^*)^{-1}$ commute. Thus
\begin{align*}
 g(A,B)^* &= -if(U_A^*,U_B)^*\left(f(U_A^*,-U_B)^{-1}\right)^*\\
          &= -if(U_A,U_B^*)f(U_A,-U_B^*)^{-1}\\
          &= -if(U_A,-U_B^*)^{-1}f(U_A,U_B^*)\\
          &= if(-U_A,U_B^*)^{-1}f(U_A,U_B^*)=g(A,B).
\end{align*}
\end{proof}

\vspace{0,2cm}

Let $A\in H_m$ and $B\in H_n$ be two Hermitian matrices. In the remainder of this section, we first answer the question (Theorem \ref{T1}): 
{\em Under what conditions can we write the Cayley transform of $A\otimes B$ as a Kronecker product $C\otimes D$ with $C\in M_m$, $D\in M_n$?}
The next natural question here is (Theorem \ref{T3}): {\em If $U_{A\otimes B}=C\otimes D$ for some $C\in M_m$, $D\in M_n$, under what conditions $C=U_A$ and $D=U_B$?}

\vspace{0,2cm}

In the following we use a variation of the result about the separability constraints which was proved in \cite{11}. 

\begin{lemma}
\label{L1}
Let $\{\, C_1,\,\ldots,\, C_m\,\}$ be a basis for an $m$-dimensional subspace of $M_m$ and
$\{\, D_1,\,\ldots,\, D_n\,\}$ be a basis for an $n$-dimensional subspace of $M_n$. Then a
matrix $A$ in the form
$$
A=\sum_{j=1}^{m}\sum_{k=1}^{n} a_{j,k} C_j\otimes D_k
$$
can be written as $A_1\otimes A_2$ with $A_1\in M_m$ and $A_2\in M_n$
if and only if
$$a_{p,q}a_{r,s}=a_{p,s}a_{r,q}$$
for all $p,r\in\{\,1,\,\ldots,\,m\,\}$ and $q,s\in\{\,1,\,\ldots,\,n\,\}$.
\end{lemma}

Now we are in the position to write our first result.

\begin{theorem}
\label{T1}
Let $A\in H_m$ and $B\in H_n$. Then the Cayley transform of $A\otimes B$ can be written as $C\otimes D\in M_m\otimes M_n$ 
if and only if one of the following conditions is fulfilled.
\begin{itemize}
	\item[(a)] $A$ has one eigenvalue.
	\item[(b)] $B$ has one eigenvalue.
	\item[(c)] $A$ has two eigenvalues $a_1,a_2$ and $B$ has two eigenvalues $b_1,b_2$ such that $a_1 b_1 a_2 b_2 = 1$.
\end{itemize}
\end{theorem}

\begin{proof}
Let $A\in H_m$ be a Hermitian matrix with eigenvalues $\{a_1,\ldots,a_m\}$ 
and corresponding orthonormal eigenvectors $\{x_1,\ldots,x_m\}$ and let $B\in H_n$ be a Hermitian matrix with eigenvalues $\{b_1,\ldots,b_n\}$ and corresponding orthonormal
eigenvectors $\{y_1,\ldots,y_n\}$. Then the Cayley transform of $A\otimes B$ can be written in the form
$$U_{A\otimes B}=\sum_{j=1}^m\sum_{k=1}^n \frac{a_j b_k-i}{a_j b_k+i} x_j x_j^*\otimes y_k y_k^*~.$$
We already know that the eigenvectors are preserved under the Cayley transform. Thus, by Lemma \ref{L1} (with identifying $C_j\to x_j x_j^*$, $D_k\to y_k y_k^*$, and
$a_{j,k}\to(a_j b_k-i)/(a_j b_k+i)$), we find that
\begin{equation}
\label{E1.1}
U_{A\otimes B}=C\otimes D
\end{equation} 
for some $C\in M_m$ and $D\in M_n$ if and only if 
$$\frac{a_p b_q-i}{a_p b_q+i}\cdot\frac{a_r b_s-i}{a_r b_s+i}= \frac{a_p b_s-i}{a_p b_s+i}\cdot\frac{a_r b_q-i}{a_r b_q+i}~$$
for all $1\le j,p\le m$ and $1\le k,q\le n$.
This equation can be rewritten as
\begin{equation}
\label{E1.2}
(a_p - a_r)(b_q - b_s)(a_pa_rb_qb_s - 1) = 0.
\end{equation}
 It follows that either $a_j b_k a_p b_q = 1$ or $a_j = a_p$ or $b_k = b_q$ must hold for all eigenvalues $a_j$, 
$a_p$ and $b_j$,$b_q$ of $A$ and $B$, respectively.

\vspace{0,2cm}

Obviously, if either (a), (b), or (c) holds, then \eqref{E1.2} is fulfilled and we are done. For the converse implication,
suppose that equation \eqref{E1.2} holds and assume that neither $A$ nor $B$ has only one eigenvalue. Let $a_p,a_r$ and $b_q,b_s$ 
be two distinct eigenvalues of $A$ and $B$ respectively. It follows from \eqref{E1.2} that $a_pa_rb_qb_s = 1$. Thus,
$a_p\neq0$, $a_r\neq 0$, $b_q\neq 0$, and $b_s\neq 0$. Let $a_{t}$ be an eigenvalue
of $A$ and assume $a_{t}$ is distinct from $a_p$ and $a_r$. 
It follows that $a_pa_tb_qb_s = a_pa_rb_qb_s = 1$ which yields $a_t = a_r$, a contradiction.
Thus, $A$ has 2 eigenvalues and they are nonsingular. Similarly $B$ has 2 eigenvalues and they are nonsingular. If we denote the eigenvalues of $A$ with $a_1, a_2$
and the eigenvalues of $B$ with $b_1, b_2$, then the equation \eqref{E1.2} reduces to $a_1a_2b_1b_2 = 1$. The proof is completed.
\end{proof}

\begin{theorem}
\label{T2}
Let $A\in H_m$ and $B\in H_n$ be such that the Cayley transform $U_{A\otimes B}=C'\otimes D'\in M_m\otimes M_n$.
Then there exist $C\in H_m$ and $D\in H_n$ such that $U_{A\otimes B}=U_C\otimes U_D$.
\end{theorem}

\begin{proof}
We may assume that $C'$ and $D'$ are unitary. Furthermore, we may assume that neither $C'$ nor $D'$
have 1 as an eigenvalue. Namely, there exists $\theta\in\mathbb{R}$ such that neither $e^{i\theta}C'$
nor $e^{-i\theta}D'$ have 1 as an eigenvalue (since $C'$ an $D'$ are finite). Thus, by the invertibility of the Cayley transform,
there exist $C\in H_m$ and $D\in H_n$ such that $U_C=C'$ and $U_D=D'$. Consequently, $U_{A\otimes B}=U_C\otimes U_D$.
\end{proof}

The next theorem answers the question when $U_{A\otimes B}=U_A\otimes U_B$.

\begin{theorem}
\label{T3}
Let $A\in H_m$ and $B\in H_n$. Then $U_{A\otimes B}=U_A\otimes U_B$ if and only if one of the following conditions is fulfilled.
\begin{itemize}
	\item[(a)] $A$ has one eigenvalue $a\neq 0$ and $B$ has one or two eigenvalues given by
                   $$b_{1,2}=\frac{a(1-a)\pm\sqrt{a^2(1-a)^2-4a}}{2a}.$$
	\item[(b)] $B$ has one eigenvalue $b\neq 0$ and $A$ has one or two eigenvalues given by
                   $$a_{1,2}=\frac{b(1-b)\pm\sqrt{b^2(1-b)^2-4b}}{2b}.$$
\end{itemize}
\end{theorem}

\begin{proof}
Let $A\in H_m$ be a Hermitian matrix with eigenvalues $\{a_1,\ldots,a_m\}$ and corresponding orthonormal eigenvectors $\{x_1,\ldots,x_m\}$ 
and let $B\in H_n$ be a Hermitian matrix with eigenvalues $\{b_1,\ldots,b_n\}$ and corresponding orthonormal eigenvectors $\{y_1,\ldots,y_n\}$.  Then 
the Cayley transform of $A$, $B$ and $A\otimes B$ can be written as
$$U_{A}=\sum_{j=1}^m \frac{a_j-i}{a_j+i} x_j x_j^*,\qquad U_{B}=\sum_{k=1}^n \frac{b_k-i}{b_k+i} y_k y_k^*,$$
$$U_{A\otimes B}=\sum_{j=1}^m\sum_{k=1}^n \frac{a_j b_k-i}{a_j b_k+i} x_j x_j^*\otimes y_k y_k^*~.$$
Comparing $U_A\otimes U_B$ and $U_{A\otimes B}$ yields
$$ \frac{a_j-i}{a_j+i}\cdot \frac{b_k-i}{b_k+i} = \frac{a_j b_k-i}{a_j b_k+i} $$
or equivalently
\begin{equation}
\label{E1.3}
 a_jb_k(1-a_j-b_k) = 1,
\end{equation}
where $1\leq j\leq m$ and $1\leq k\leq n$. So, $U_{A\otimes B}=U_A\otimes U_B$ if and only if the relation (\ref{E1.3}) holds for all $1\leq j\leq m$ and $1\leq k\leq n$.

Now, if either (a) or (b) holds, then it is easy to see that \eqref{E1.3} is fulfilled and we are done. For the converse implication, suppose that the relation 
\eqref{E1.3} holds for all $1\leq j\leq m$ and $1\leq k\leq n$. 
Since we assumed that $U_{A\otimes B}=U_A\otimes U_B$, 
Theorem \ref{T1} implies three cases.

\vspace{0,2cm}

\noindent 
{\em Case 1.}
Suppose that $A$ has one eigenvalue $a$. Then $$ ab_k^2-a(1-a)b_k+1 = 0 $$ for all $1\leq k\leq n$. Note that $a\ne 0$. Thus, it is easy to compute that
$$b_k=\frac{a(1-a)\pm\sqrt{a^2(1-a)^2-4a}}{2a}.$$ In particular, $B$ has one or two eigenvalues as in the case (a).

\vspace{0,2cm}

\noindent 
{\em Case 2.}
If $B$ has one eigenvalue, then, using the same arguments as in the previous case, we obtain (b).

\vspace{0,2cm}

\noindent 
{\em Case 3.}
Finally, suppose that $A$ has two eigenvalues $a_1, a_2$ and $B$ has two eigenvalues $b_1, b_2$ such that $a_1b_1a_2b_2=1$ and which satisfy \eqref{E1.3}
for all $j,k\in\{1,2\}$. However, no such $a_1,a_2,b_1,b_2\in\mathbb{R}$ exist. Namely, from \eqref{E1.3} we have the following four equations
\begin{align*}
		 a_1b_1(1-(a_1+b_1))&=1,&
		 a_1b_2(1-(a_1+b_2))&=1,\\
		 a_2b_1(1-(a_2+b_1))&=1,&
		 a_2b_2(1-(a_2+b_2))&=1.
\end{align*}
Using $a_1b_1a_2b_2=1$, we obtain
\begin{align*}
\label{E1.4}
		 (1-(a_1+b_1))&=a_2b_2,&
		 (1-(a_1+b_2))&=a_2b_1,\\
		 (1-(a_2+b_1))&=a_1b_2,&
		 (1-(a_2+b_2))&=a_1b_1.
\end{align*}
Thus, we find
\begin{align*}
 1&=a_2b_2a_1b_1=(1-(a_1+b_1))(1-(a_2+b_2))\\
  &=a_2b_1a_1b_2=(1-(a_1+b_2))(1-(a_2+b_1)).
\end{align*}
It follows that
$$(a_1-a_2)(b_1-b_2)=0$$
which cannot be satisfied since $a_1\neq a_2$ and $b_1\neq b_2$.
Hence, this case does not yield any solutions.
\end{proof}

\begin{corollary}
\label{C1}
For all $A\in H_m$, we have $U_{A\otimes A}\ne U_A\otimes U_A$.
\end{corollary}

\begin{proof}
If $U_A \otimes U_A = U_{A \otimes A}$ then, by Theorem \ref{T3},
$A = a I_m$ for some $a\in\mathbb{R}\setminus\{0\}$ with 
\begin{equation*}
 a = \frac{a(1-a)}{2a} \qquad  {\rm and} \qquad a^2(1-a)^2-4a = 0,
\end{equation*}
which has no solutions for $a\in\mathbb{R}$.
\end{proof}

\begin{corollary}
\label{C2}
For all $B\in H_n$, we have $U_{I_m\otimes B}\ne U_{I_m}\otimes U_B$. Similarly,
for all $A\in H_m$, we have $U_{A\otimes I_n}\ne U_A\otimes U_{I_n}$.
\end{corollary}


\begin{remark}\em
\label{R1}
The natural question here is whether the analogue results hold true if we replace the Kronecker product with some other product, for example, with the star product 
or with the direct sum. In particular, if $A\in H_m$ and $B\in H_n$, then one can easily show that $U_{A\oplus B} = U_{A}\oplus U_B$. Now, recall that the star product 
of $A=(a_{jk})\in M_2$ and $B\in M_n$ is defined by
$$A \star B := 
\begin{pmatrix}
a_{11} & 0_{1\times n}  & a_{12} \\ 
0_{n\times 1} & B & 0_{n\times 1} \\
a_{21} & 0_{1\times n} & a_{22}
\end{pmatrix},$$
where $0_{n\times 1}$ is a column of $n$ zeros, and $0_{1\times n}=0_{n\times 1}^t$. It is also known that there exists a permutation matrix $P\in M_{n+2}$ such that 
$P(A\star B)P^t=A\oplus B$. Therefore, if $A\in H_2$ and $B\in H_n$, then $U_{A\star B} = U_{A}\star U_B$ as well.
\end{remark}

\begin{remark}\em
\label{R2}
At the end, let us point out that we have considered just the bipartite case, i.e., $M_m\otimes M_n$ with integers $m,n\ge 2$. But we can naturally extend our results
to the  multipartite systems $M_{n_1}\otimes\cdots \otimes M_{n_m}$ with $n_1,\ldots, n_m\ge 2$ and $m>2$ when
$$
 U_{M_{n_j}\otimes\cdots\otimes M_{n_m}}
  =U_{M_{n_j}}\otimes U_{M_{n_{j+1}}\otimes\cdots\otimes M_{n_m}}\qquad\forall j\in\{1,\ldots,m-1\}
$$
or
$$
 U_{M_{n_1}\otimes\cdots\otimes M_{n_j}}
  =U_{M_{n_1}}\otimes U_{M_{n_2}\otimes\cdots\otimes M_{n_j}}\qquad\forall j\in\{2,\ldots,m\}.
$$
Namely, if, for example, $m=3$ and $A_1\in H_{n_1}$, $A_2\in H_{n_2}$,
$A_3\in H_{n_3}$, then we first observe when $U_{A_1\otimes (A_2\otimes A_3)} = U_{A_1}\otimes U_{A_2\otimes A_3}$. According to Theorem \ref{T3}, this is true if and only if
one of the following conditions is fulfilled.
\begin{itemize}
	\item[(a)] $A_1$ has one eigenvalue (which is nonsingular) and $A_2\otimes A_3$ has one or two eigenvalues (which are nonsingular) given by the exact formula (see case (a) in Theorem \ref{T3}).
	\item[(b)] $A_2\otimes A_3$ has one eigenvalue (which is nonsingular) and $A_1$ has one or two eigenvalues (which are nonsingular) given by the exact formula (see case (b) in Theorem \ref{T3}).
\end{itemize}
Again, using Theorem \ref{T3}, we find out that $U_{A_1\otimes A_2\otimes A_3} = U_{A_1}\otimes U_{A_2}\otimes U_{A_3}$ if
for distinct $j,k,l\in \{1,2,3\}$, $A_j$ and $A_k$ have one eigenvalue (nonsingular) and $A_l$ has one or two eigenvalues (nonsingular eigenvalues are given by the exact 
formulas). Similarly, $U_{A_1\otimes\cdots\otimes A_m} = U_{A_1}\otimes\cdots\otimes U_{A_m}$ for Hermitian matrices $A_1\in H_{n_1},\ldots, A_m\in H_{n_m}$, $m>3$, if
all the matrices $A_1,\ldots, A_m$, with one possible exception (this exception may have two eigenvalues, both nonsingular), have 
one eigenvalue (nonsingular) given by the exact formula. We omit the details since the proofs are rather technical.
\end{remark}

\begin{remark}\em
\label{R3}
The results in the bipartite case do not extend to the multipartite case in
a straightforward way. For example, we have
\begin{align*}
 U_{I_m\otimes I_m}&\neq U_{I_m}\otimes U_{I_m},\\
 U_{I_m\otimes I_m\otimes I_m}&\neq U_{I_m}\otimes U_{I_m}\otimes U_{I_m},\\
 U_{I_m\otimes I_m\otimes I_m\otimes I_m}&\neq U_{I_m}\otimes U_{I_m}\otimes U_{I_m}\otimes U_{I_m}
\end{align*}
but
\begin{equation*}
 U_{I_m\otimes I_m\otimes I_m\otimes I_m\otimes I_m}=U_{I_m}\otimes U_{I_m}\otimes U_{I_m}\otimes U_{I_m}\otimes U_{I_m}.
\end{equation*}
\end{remark}


\-

\noindent{\bf Acknowledgments.}
Our thanks to an anonymous referee for raising the question whether there
is an analogue for the Kronecker sum (under exponentiation) for the Cayley transform.
Y. Hardy and W.-H. Steeb are supported by the National Research Foundation (NRF),
South Africa. This work is based upon research supported by the National
Research Foundation. Any opinion, findings and conclusions or recommendations
expressed in this material are those of the author(s) and therefore the
NRF do not accept any liability in regard thereto.



\end{document}